\documentclass[12pt,a4paper]{scrartcl}
\usepackage{amssymb,amsmath,amsthm}

\usepackage[all,cmtip]{xy}
\usepackage[american]{babel}
\usepackage[T1]{fontenc}
\usepackage[utf8]{inputenc}
\usepackage{url}
\usepackage{xspace}
\usepackage{graphicx}
\usepackage{color}
\usepackage{subfig}
\usepackage{setspace}
\theoremstyle{plain}

\newtheorem{mtheo}{Theorem}		%main theorem: counted within the document
\newtheorem{mcoro}[mtheo]{Corollary}	%main corollary: counted within the document, same counter as mtheo
\newtheorem*{theorem*}{Theorem}
\newtheorem{lemma}{Lemma}[section]			%lemma: counted within the section
\newtheorem{prop}[lemma]{Proposition}			%proposition: counted within the section, same counter as lemma

\theoremstyle{definition}
\newtheorem{rem}[lemma]{Remark}
                               
\DeclareMathOperator{\Aut}{Aut}
\DeclareMathOperator{\PGL}{PGL}
\DeclareMathOperator{\PSL}{PSL}
\DeclareMathOperator{\rank}{rank}
\newcommand{\V}{\ensuremath{\mathrm{V}}}
\newcommand{\gapcom}[1]{\texttt{#1}\xspace}
\newlength{\heightofhw}
\newcommand{\eigbox}[2]{\ensuremath{#1 \times \boxed{\rule{0cm}{\heightofhw} #2}}}

\title{Rational conjugacy of torsion units\\ in integral group rings\\ of non-solvable groups\footnotetext{ \emph{$2010$ Mathematics Subject Classification.} 16U60, 16S34, 20C05, 20C10\\ \noindent \emph{Key words and phrases.} Integral group ring, torsion unit, Zassenhaus Conjecture, Prime graph question.}}

\author{Andreas Bächle, Leo Margolis}
\date{}

\begin{document}
\maketitle

\begin{abstract} \textsc{Abstract.} We introduce a new method to study rational conjugacy of torsion units in integral group rings using integral and modular representation theory. Employing this new method, we verify the first Zassenhaus Conjecture for the group $\textup{PSL}(2,19)$. We also prove the Zassenhaus Conjecture for $\textup{PSL}(2,23)$. In a second application we show that there are no normalized units of order $6$ in the integral group rings of $M_{10}$ and $\PGL(2,9)$. This completes the proof of a theorem of W.~Kimmerle and A.~Konovalov that the Prime Graph Question has an affirmative answer for all groups having an order divisible by at most three different primes. \end{abstract}

Througout this paper let $G$ be a finite group, $\mathbb{Z}G$ the integral group ring of $G$ and $\V(\mathbb{Z}G)$ the group of augmentation one units in $\mathbb{Z}G$, aka.\ normalized units. The most famous open conjecture regarding torsion units in $\mathbb{Z}G$ is:\\

\noindent \textbf{The Zassenhaus Conjecture (ZC):} Let $u \in \V(\mathbb{Z}G)$ be a torsion unit. Then there exist a unit $x \in \mathbb{Q}G$ and $g \in G$ such that $x^{-1}ux = g.$\\

If for a unit $u$ such $x$ and $g$ exist we say that $u$ is rationally conjugate to $g$. There are positive results for the Zassenhaus Conjecture for classes of solvable groups (e.g.\ A.~Weiss proved it for nilpotent groups \cite{WeissCrelle} and M.~Caicedo, L.~Margolis and Á.~del~Río established it for all cyclic-by-abelian groups \cite{CyclicByAbelian}). For non-solvable groups it is only known for specific groups, e.g.\ for $A_5$ \cite{LutharPassiA5}, $S_5$ \cite{LutharTrammaComm}, $A_6$ \cite{HertweckA6}, or $\PSL(2,p)$ for $p\leq 17$ a prime \cite{HertweckBrauer, KiKoStAndrews, Gildea}.

Considering the difficulty of the Zassenhaus Conjecture and motivated by the results in \cite{Ki2006} it was proposed in \cite[Problem 21]{Ari} to study first the following question:\\

\noindent \textbf{ The Prime Graph Question (PQ):} Let $p$ and $q$ be different primes such that $\V(\mathbb{Z}G)$ has a unit of order $pq$. Does this imply that $G$ has an element of that order?\\

This is the same as to ask whether $G$ and $\V(\mathbb{Z}G)$ have the same prime graph. Much more is known here: e.g.\ it has an affirmative answer for all solvable groups \cite{Ki2006} or the series $\PSL(2,p)$, $p$ a prime \cite{HertweckBrauer}. V.~Bovdi and A.~Konovalov with different collaborators obtained positive answers to (PQ) for many of the sporadic simple groups, see e.g.\ \cite{KonovalovM24} for recent results. Lately substantial progress was made, when W.~Kimmerle and A.~Konovalov obtained the first reduction result for the Prime Graph Question \cite[Proposition 4.1]{KiKoStAndrews} (cf.\ also \cite[Theorem 2.1]{KiKo}):
\begin{theorem*}[Kimmerle, Konovalov] If (PQ) has affirmative answer for all almost-simple images of $G$, then it also has an affirmative answer for $G$ itself. \end{theorem*}

Recall that a group $A$ is almost simple if $S \leq A \leq \Aut(S)$ for a simple group $S$. Using the above theorem, they proved that the Prime Graph Question has a positive answer for all finite  groups whose order is divisible by at most three different primes, if it has a positive answer for $M_{10}$ and $\PGL(2,9)$ \cite[Theorem 3.1]{KiKo}. Their result also places special emphasis on investigating the Prime Graph Question for almost simple groups.

All proofs of (ZC) for non-solvable groups rely on the so-called Luthar-Passi-Hertweck-method \cite{LutharPassiA5,HertweckBrauer}, referred to as the HeLP-method. But in many cases this method does not suffice to prove (ZC), e.g. it fails for $A_6$ \cite{HertweckA6}, $\PSL(2,19)$ (see below) or $M_{11}$ \cite{KonovalovM11}. Sometimes special arguments were considered in such situations as in \cite{LutharTrammaComm}, \cite[Example 2.6]{HertweckAlgColloq}, or \cite{HertweckA6}. But these arguments were designed for very special situations and are hard to generalize or seem not to give new information in other situations.

In this paper we introduce a new method to study rational conjugacy of torsion units inspired by M.~Hertweck's arguments for proving (ZC) for the alternating group of degree 6 \cite{HertweckA6}. This method is especially interesting for units of mixed order (i.e.\ not of prime power order) and in combination with the HeLP-method. We then give two applications of this method to prove: 

\begin{mtheo} The Zassenhaus Conjecture holds for the groups $\PSL(2,19)$ and $\PSL(2,23).$\end{mtheo}
((ZC) for $\PSL(2,23)$ is proved using known methods.)

\begin{mtheo}\label{theo2} There are no units of order 6 in $\V(\mathbb{Z}M_{10})$ and in $\V(\mathbb{Z}\PGL(2,9)).$ Here $M_{10}$ denotes the Mathieu group of degree 10.\end{mtheo}

Theorem \ref{theo2} together with \cite[Theorem 2.1, Theorem 3.1]{KiKo} (or \cite[\S 4]{KiKoStAndrews}) directly yields:

\begin{mcoro} Let $G$ be a finite group. Suppose that the order of all almost simple images of $G$ is divisible by at most three different primes. Then the prime graph of the normalized units of $\mathbb{Z}G$ coincides with that of $G$. In particular, the Prime Graph Question has a positive answer for all groups with an order divisible by at most three different primes.\end{mcoro}

\section{From eigenvalues under ordinary representations to the modular module structure}

Let $G$ be a finite group. The main tool to study rational conjugacy of torsion units are partial augmentations: Let $u = \sum\limits_{g \in G} a_g g \in \mathbb{Z}G$ and $x^G$ be the conjugacy class of the element $x \in G$ in $G.$ Then $\varepsilon_x(u) = \sum\limits_{g \in x^G} a_g$ is called the partial augmentation of $u$ at $x.$ This relates to (ZC) via:

\begin{lemma}[Marciniak, Ritter, Sehgal, Weiss] Let $u \in \V(\mathbb{Z}G)$ be a torsion unit of order $n$. Then $u$ is rationally conjugate to a group element if and only if $\varepsilon_x(u^d) \geq 0$ for all $x \in G$ and all powers $u^d$ of $u$ with $d \mid n$. \end{lemma}

\noindent\emph{Proof.} \cite[Theorem 2.5]{MarciniakRitterSehgalWeiss}.\\

It is well known that if $u \neq 1$ is a torsion unit in $\V(\mathbb{Z}G)$, then $\varepsilon_1(u)=0$ by the so-called Berman-Higman Theorem \cite[Proposition 1.4]{SehgalBook2}. If $\varepsilon_x(u) \neq 0$, then the order of $x$ divides the order of $u$ \cite[Theorem 2.7]{MarciniakRitterSehgalWeiss}, \cite[Proposition 3.1]{HertweckAlgColloq}. Moreover the exponents of $G$ and of $\V(\mathbb{Z}G)$ coincide \cite{CohnLivingstone}. We will use this in the following without further mention.

Let $K$ be a field, $D$ a $K$-representation of $G$ with corresponding character $\chi$ and $u \in \V(\mathbb{Z}G)$ a torsion unit of order $n$. If $\chi$ and all partial augmentations of $u$ and all its powers are known, and the characteristic of $K$ does not divide $n$, we can compute the eigenvalues of $D(u)$ in a field extension of $K$ which is large enough (a field which is a splitting field for $G$ and all its subgroups will do; there will be plenty of examples for this kind of calculations in §2). The HeLP-method makes use of the fact that the multiplicity of each $n$-th root of unity as an eigenvalue of $D(u)$ is a non-negative integer.\\

\noindent \textbf{Notations:} $p$ will always denote a prime, $\mathbb{Q}_p$ the $p$-adic completion of $\mathbb{Q}$ and $\mathbb{Z}_p$ the ring of integers of $\mathbb{Q}_p$. By $R$ we denote a complete local ring with maximal ideal $P$ containing $p$ and by $K$ the field of fractions of $R.$ Denote by $k$ a finite field of characteritic $p$ containing the residue class field of $R$. The reduction modulo $P$, also with respect to lattices, will be denoted by \ $\bar{}$ .\\ 

The idea of our method is that if $D$ is an $R$-representation of a group $G$ with corresponding $RG$-lattice $L$ and $u$ is a torsion unit in $\mathbb{Z}G$ of order divisible by $p$, we can reduce $D$ modulo $P$ and obtain restrictions on the isomorphism type of $kG$-modules considered as $k\langle \bar{u} \rangle$-modules, where $k$ is big enough to allow realizations of all irreducible $p$-modular representations of $G.$ Note that the Krull-Schmidt-Azumaya Theorem holds for finitely generated $RG$-lattices \cite[Theorem 30.6]{CurtisReiner1}. From these isomorphism types we can then obtain restrictions on the isomorphism types of the $kG$-composition factors of $\bar{L}$ when viewed as $k\langle \bar{u} \rangle$-module. Since a simple $kG$-module may appear in the reduction of several ordinary representations, this may finally yield a contradiction to the existence of $u.$ A rough sketch of the method is given in Figure \ref{idea}. As the direct path from the eigenvalues of an ordinary representation to the isomorphism types of the corresponding reduced module is not always evident, we are sometimes forced to take take the detour along the dashed arrows.

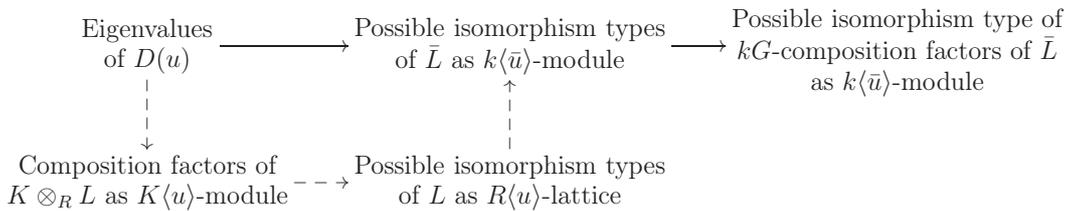
\begin{figure}[h]
\xyoption{frame}
\begin{displaymath}
   \resizebox{.9\textwidth}{!}{ \xymatrix{ 
    *+{\txt{Eigenvalues\\ of $D(u)$}}*\frm<8pt>{-} \ar@{->}[r] \ar@{-->}[d] &
    *+{\txt{Possible isomorphism types\\ of $\bar{L}$ as $k\langle\bar{u}\rangle$-module}} \ar[r] & \txt{Possible isomorphism type of \\ $kG$-composition factors of $\bar{L}$ \\ as $k\langle \bar{u} \rangle$-module} \\
    *+{\txt{Composition factors of\\ $K\otimes_R L$ as $K\langle u \rangle$-module}} \ar@{-->}[r] &
    *+{\txt{Possible isomorphism types\\ of $L$ as $R\langle u\rangle$-lattice}} \ar@{-->}[u]  }
    }
\end{displaymath}
\caption{The idea of the method.}\label{idea}
\end{figure}

The connections between the eigenvalues of ordinary representations and the isomorphism type of the modules in positive characteristic for some cases are contained in the following propositions, which are consequences of known modular and integral representation theory.

The first proposition is standard knowledge in modular representation theory and may be found in e.g.\ \cite[Theorem 5.3, Theorem 5.5]{Huppert2}.

\begin{prop}\label{prop_modules_cyclic_mod_p}  Let $G = \langle g \rangle$ be a cyclic group of order $p^am$, where $p$ does not divide $m$. Let $k$ be a field of characteristic $p$ containing a primitive $m$-th root of unity $\xi$. Then:
\begin{itemize}
\item[a)] Up to isomorphism there are $m$ simple $kG$-modules. All these modules are one-dimensional as $k$-vector spaces, $g^m$ acts trivially on them and $g^{p^a}$ acts as $\xi^i$ for $1 \leq i \leq m$. We denote these modules by $k^{\xi}, k^{\xi^2}, \ ..., \ k^{\xi^m}$.
\item[b)] The projective indecomposable $kG$-modules are of dimension $p^a$. They are all uniserial and all composition factors of a projective indecomposable $kG$-module are isomorphic. There are $m$ non-isomorphic projective indecomposable $kG$-modules.
\item[c)] Each indecomposable $kG$-module is isomorphic to a submodule of a projective indecomposable module. So there are $p^am$ indecomposable modules, which are all uniserial and all composition factors of an indecomposable $kG$-module are isomorphic.  
\end{itemize}\end{prop}

Using Proposition \ref{prop_modules_cyclic_mod_p} and the fact that idempotents may be lifted \cite[Theorem 30.4]{CurtisReiner1} we obtain:

\begin{prop}\label{prop_lattices_cyclic_group} Let $G = \langle g \rangle$ be a cyclic group of order $p^am$, where $p$ does not divide $m$. Let $R$ be a complete local ring containing a primitive $m$-th root of unity $\zeta$. Let $D$ be an $R$-representation of $G$ and let $L$ be an $RG$-lattice affording this representation.\\
Let $A_i$ be sets with multiplicities of $p^a$-th roots of unity such that $\zeta A_1 \cup \zeta^2 A_2 \cup ... \cup \zeta^mA_m$ are the complex eigenvalues of $D(g),$ where $A_i = \emptyset$ is possible. Let $V_1, \ ..., \ V_m$ be $KG$-modules such that if $E_i$ is a representation of $G$ affording $V_i$ the eigenvalues of $E_i(g)$ are $\zeta^iA_i.$ Then 
\[L \cong L^{\zeta^1} \oplus ... \oplus L^{\zeta^m} \quad \text{and} \quad \bar{L} \cong \bar{L}^{\zeta^1} \oplus ... \oplus \bar{L}^{\zeta^m}\]
such that $\rank_R(L^{\zeta^i})=\dim_k(\bar{L}^{\zeta^i}) = |A_i|.$ (The superscripts $\zeta^i$ are merely meant as indices.) Moreover $K \otimes_R L^{\zeta^i} \cong V_i$ and the only composition factor of $\bar {L}^{\zeta^i}$ is $k^{\xi^i}$, see the notation in Proposition \ref{prop_modules_cyclic_mod_p}.\end{prop}

To understand the full connection between the eigenvalues $\zeta^iA_i$ and the structure of $L^{\zeta^i}$, i.e.\ to follow the arrow in the second line of Figure \ref{idea}, one must study the representation theory of $R\langle g^m \rangle.$ The representation type of $R\langle g^m \rangle$ may be finite, tame or wild. Roughly speaking, the representation theory gets more complicated with increasing $a$ and increasing ramification index of $K$ over $\mathbb{Q}_p.$ A listing of all representation types may be found in \cite{Dieterich}. Some results concerning the connection between $A_i$ and $L^{\zeta^i}$ are recorded in the next propositions. The first one is a consequence of \cite[Theorem 2.6]{HellerReiner1}.

\begin{prop}\label{prop_lattices_cyclic_group_of_prime_order} Let the notation be as in Proposition \ref{prop_lattices_cyclic_group}, assume $G \cong C_p$ and that $K$ is unramified over $\mathbb{Q}_p$. Let $\gamma$ be a primitive $p$-th root of unity. Up to isomorphism there are 3 indecomposable $RG$-lattices $M_1, M_2, M_3$. Each $\bar{M}_i$ remains indecomposable. The $R$-rank and the corresponding eigenvalues of $D(g)$ are: $\rank_R(M_1) = 1 $ with eigenvalue $1$, $\rank_R(M_2) = p-1$ with eigenvalues $\gamma, \gamma^2, ..., \gamma^{p-1}$ and $\rank_R(M_3) = p$ with eigenvalues $1, \gamma, \gamma^2, ..., \gamma^{p-1}$. \end{prop}

\noindent \textbf{Notation:} We denote the indecomposable lattices in Proposition \ref{prop_lattices_cyclic_group_of_prime_order} with their natural names: the trivial lattice $M_1 = R$, the augmentation ideal $M_2 = I(RC_p)$ and the group ring $M_3 = RC_p$.\\
When considering $R C_{2p}$-lattices or $k C_{2p}$-modules as in Proposition \ref{prop_lattices_cyclic_group} (i.e.\ $a = 1$ and $m = 2$), we abbreviate the superscript $1$ and $-1$ to $+$ and $-$, respectively, i.e. $L^+ = L^{1}$, $L^- = L^{-1}$, $\bar{L}^+ = \bar{L}^{1}$ and $\bar{L}^- = \bar{L}^{-1}$ for the direct summands having trivial and non-trivial composition factors, respectively.

\begin{prop}\label{num_of_components}
Let the notation be as in Proposition \ref{prop_lattices_cyclic_group}, $p$ be odd and $p \equiv \epsilon \mod 4$ with $\epsilon \in \{\pm 1\}.$ Assume $K$ is the $p$-adic completion of an extension of $\mathbb{Q}(\sqrt{\epsilon p})$ which is unramified at $p$. Denote by $R$ the ring of integers of $K$ and let $G \cong C_p$. Note that there are exactly three simple $RG$-lattices up to isomorphism.
%Assume $K '= \mathbb{Q}_p\left(\sqrt{\epsilon p}\right)$ and $G \cong C_p$. Moreover let $R'$ be the ring of integers of $K'.$ Then all indecomposable $R'G$-lattices are given in \cite[Lemma 4.1]{Gudivok} and up to isomorphism there are $2p+3$ of them. Note that there are exactly three simple $R'G$-lattices up to isomorphism. \\
%If $K$ is an unramified extension of $K'$ with ring of intgers $R$, there are still exactly $2p+3$ indecomposable $RG$-modules by \cite[3.2]{Dieterich1}. So these are the same as in \cite[Lemma 4.1]{Gudivok}. 
The two following facts hold for indecomposable $RG$-lattices:   
\begin{enumerate}
\item[a)] If $L$ is an indecomposable $RG$-lattice, then the non-trivial simple lattices each appear at most once as a composition factor of L, and the trivial one at most twice.
\item[b)] If $L$ is an indecomposable $RG$-lattice having at most two non-isomorphic composition factors, then each composition factor appears at most once.
\end{enumerate}
\end{prop}

\begin{proof}
If $K'$ is the $p$-adic completion of $\mathbb{Q}(\sqrt{\epsilon p})$ and $R'$ is its ring of integers, then up to isomorphism all indecomposable $R'G$-lattices are explicitly given in \cite[Lemma 4.1]{Gudivok}. These lattices satisfy the statements of the proposition. If $L$ is an indecomposable $RG$-lattice, it is a direct summand of $R \otimes_{R'} L'$, where $L'$ is an indecomposable $R'G$-lattice, by the last paragraph of the proof of \cite[Proposition 33.16]{CurtisReiner1}. So the statements of the proposition still hold for indecomposable $RG$-lattices. 
% By \cite[3.2]{Dieterich1} there are also precisly $2p+3$ indecomposable $RG$-lattices and since $R' \subset R$ they must, up to isomorphism, be the same as the indecomposable $R'G$-lattices or sublattices of these. The statements of the proposition are then still satisfied. 
\end{proof}

\section{Applications}

For a group $G$ we denote by $\chi_i$ an ordinary character of $G$ and by $D_i$ a representation of $G$ affording this character. By $\varphi_i$ we denote a Brauer character and by $\Theta_i$ a representation affording $\varphi_i.$ We write $D_i(u) \sim (\alpha_1,...,\alpha_j)$ or $\Theta_i(u) \sim (\alpha_1,...,\alpha_j)$ to indicate that $\alpha_1,...,\alpha_j$ are the eigenvalues (with multiplicities) of the corresponding matrix. To improve readability, we sometimes group the eigenvalues appearing several times, e.g.\ $D_i(u) \sim (\eigbox{3}{1}, \eigbox{2}{\zeta,\zeta^{-1}})$ indicates that $D_i(u)$ has $1$ three times and $\zeta$ and $\zeta^{-1}$ each twice as eigenvalue. By $\zeta_n$ we will denote some fixed primitive complex $n$-th root of unity. Especially we will use $\zeta_n$ to denote the eigenvalues of a matrix of finite order $n$ over a field of characteristic $p,$ where $p$ is prime to $n,$ in the sense of Brauer. 

Let $K$ be an algebraically closed field, $D$ a $K$-representation of $G$ with character $\chi$ and $u$ a torsion unit in $\V(\mathbb{Z}G)$ such that the characteristic of $K$ does not divide the order of $u$. Let $m$ and $n$ be natural numbers such that $u^{m+n} = u$. Let $D(u^m) \sim (\alpha_1,...,\alpha_\ell)$ and $D(u^n) \sim (\beta_1,...,\beta_\ell).$ As $D(u^m)$ and $D(u^n)$ are simultaneously diagonalizable over $K$ this means $D(u) \sim (\alpha_1\beta_{i_1},...,\alpha_\ell\beta_{i_\ell})$ with $\{i_1,...,i_\ell\} = \{1,...,\ell\}.$ On the other hand $\chi(u) = \sum\limits_{x^G} \varepsilon_x(u)\chi(x)$, where the sum runs over all conjugacy classes $x^G$ of $G$. Comparing these two computations is the basic idea of the HeLP-method. We will use it freely in the following computations.\\

\subsection{The groups $\PSL(2,q)$. Proof of Theorem 1.}
For the rest of the section let $p$ be a prime. Rational conjugacy of torsion units in integral group rings of the groups $\PSL(2,p^f)$ were studied by Hertweck in \cite{HertweckBrauer}. Combining some propositions from that note we directly obtain:

\begin{prop}[Hertweck]\label{Hertweck} Let $G = \PSL(2,p^f)$ and $u$ a torsion unit in $\V(\mathbb{Z}G)$.
\begin{itemize}
\item[a)] If $u$ is of order prime to $p,$ there exists an element in $G$ of the same order as $u.$ If moreover the order of $u$ is prime, $u$ is rationally conjugate to an element in $G.$
\item[b)] If $f=1$ and $p$ divides the order of $u$, then $u$ is of order $p$ and rationally conjugate to an element in $G.$
\item[c)] Assume $p$ is neither 2 nor 3 and $u$ is of order 6. Then $u$ is rationally conjugate to an element in $G$. 
\end{itemize}
\end{prop}

\noindent\emph{Proof.} \cite[Propositions 6.1, 6.3, 6.4, 6.6 and 6.7]{HertweckBrauer}.\\

The HeLP-method verifies the Zassenhaus Conjecture for $\PSL(2,p)$ if $p \leq 17$. We give a quick account: (ZC) is solved for $ p = 2$ already in \cite{HughesPearson}, $p = 3$ in \cite{AllenHobby}, $p=5$ in \cite{LutharPassiA5}, $p = 7$ in \cite{HertweckAlgColloq}, $p \in \{11, 13\}$ in \cite{HertweckBrauer} and $p =  17$ independently in \cite{KiKoStAndrews} and \cite{Gildea}. The HeLP-method also suffices to prove (ZC) for $p = 23$ (see below), but not for $p =19.$ We will always use the character tables and Brauer tables from the ATLAS \cite{ATLAS}\footnote{All tables used are accessible in \gapcom{GAP} \cite{GAP4} via the commands \gapcom{CharacterTable("G");} and \gapcom{CharacterTable("G") mod p;}, where \gapcom{G} is the name of the group, e.g.\ \gapcom{PSL(2,19)} or \gapcom{M10}. The corresponding decomposition matrix for a Brauer table can then be obtained by \gapcom{DecompositionMatrix}.}. We will use througout the \gapcom{GAP} notation for conjugacy classes.

For $G = \PSL(2,p^f)$ and $p>2$ we have $|G| = \frac{(p^f-1)p^f(p^f+1)}{2}$, there are cyclic subgroups of order $\frac{p^f-1}{2}$, $p$ and $\frac{p^f+1}{2}$ in $G$, and every cyclic subgroup of $G$ lies in a conjugate of such a subgroup. There are two conjugacy classes of elements of order $p$ and, if $g$ is an element of order prime to $p$, the only conjugate of $g$ in $\langle g \rangle$ is $g^{-1}$. All this follows from a result of Dickson \cite[Satz 8.27]{Huppert1}. \\

Separating the knwon HeLP-method and the new method we first list the results obtainable only using the HeLP-method.

\begin{lemma}\label{HeLPPSL}
Let $p$ be a prime,  $f \in \mathbb{N}$ and set $G=\PSL(2,p^f).$
\begin{itemize}
\item[a)] If $p$ is neither $2$ nor $3$, then elements in $\V(\mathbb{Z}G)$ of order $4, \ 9$ or $12$ are rationally conjugate to group elements. 
\item[b)] If $p$ is neither $2$ nor $5$ and $u \in \V(\mathbb{Z}G)$ is of order 10, then either $u$ is rationally conjugate to a group element or it has the following partial augmentations: \[(\varepsilon_{\gapcom{2a}}(u), \varepsilon_{\gapcom{5a}}(u),\varepsilon_{\gapcom{5b}}(u),\varepsilon_{\gapcom{10a}}(u),\varepsilon_{\gapcom{10b}}(u)) = (0,1,-1,1,0),\]
if $u^2$ is rationally conjugate to an element in $\gapcom{5a}$, or
\[(\varepsilon_{\gapcom{2a}}(u), \varepsilon_{\gapcom{5a}}(u),\varepsilon_{\gapcom{5b}}(u),\varepsilon_{\gapcom{10a}}(u),\varepsilon_{\gapcom{10b}}(u)) = (0,-1,1,0,1),\]
if $u^2$ is rationally conjugate to an element in $\gapcom{5b}$. The conjugacy classes are listed in a way, such that squares of elements in $\gapcom{10b}$ are laying in $\gapcom{5a}$.
\end{itemize}
\end{lemma}
\begin{proof}
We will use the representations given in \cite{HertweckBrauer} and explicitly proved in \cite[Lemma 1.2]{SylowLeo}, i.e.: If $a$ is an element of order $\frac{p^f+1}{2}$ and $b$ is an element of order $\frac{p^f-1}{2}$ in $G$, then there is a primitive $\frac{p^f+1}{2}$-th root of unity $\alpha$ and a primitive $\frac{p^f-1}{2}$-th root of unity $\beta$ such that for every $i \in \mathbb{N}_0$ there exists a $p$-modular representation $\Theta_i$ of $G$ with character $\varphi_i$ such that
\begin{align*}
\Theta_i(a) \sim (1,\alpha,\alpha^{-1},\alpha^2,\alpha^{-2},...,\alpha^i,\alpha^{-i}) \\
\Theta_i(b) \sim (1,\beta,\beta^{-1},\beta^2,\beta^{-2},...,\beta^i,\beta^{-i}).
\end{align*}
For convenience, the relevant parts of the characters $\varphi_i$ are collected in Tables \ref{bratabPSL_2_23_23} and \ref{bratabPSL_2_19_19} (dots indicate zeros).

Let $p \not\in \{2, 3\}$. If $u \in \V(\mathbb{Z}G)$ is of order 4, then $\varepsilon_{\gapcom{2a}}(u) = 0$ by \cite[Proposition 6.5]{HertweckBrauer}. Thus $\varepsilon_{\gapcom{4a}}(u) = 1$ and $u$ is rationally conjugate to a group element. 
  
Assume $u$ is of order 9. Then $\varepsilon_{\gapcom{3a}}(u) = 0$ by \cite[Proposition 6.5]{HertweckBrauer}. Let $\gapcom{9a}, \gapcom{9b}$ and $\gapcom{9c}$ be the conjugacy classes of elements of order $9$ in $G$ such that if $x \in \gapcom{9a}$, we have $x^2 \in \gapcom{9b}$ and $x^4 \in \gapcom{9c}$. Then \[\varepsilon_{\gapcom{9a}}(u) + \varepsilon_{\gapcom{9b}}(u) + \varepsilon_{\gapcom{9c}}(u) = 1.\]
Let $\zeta$ be a primitive complex $9$-th root of unity such that $\Theta_1(x) \sim (1, \zeta, \zeta^8).$ Since $\Theta_1(u^3) \sim (1,\zeta^3,\zeta^6)$ and $\varphi_1$ is real valued, we get $\Theta_1(u) \sim (1,\gamma,\delta)$ with $(\gamma,\delta) \in \{(\zeta,\zeta^8),(\zeta^2,\zeta^7),(\zeta^4,\zeta^5)\}.$ So
\begin{align*}(1+\zeta + \zeta^8)\varepsilon_{\gapcom{9a}}(u) + (1+\zeta^2+\zeta^7)\varepsilon_{\gapcom{9b}}(u) + (1+\zeta^4+\zeta^5)\varepsilon_{\gapcom{9c}}(u) \\ \in \{\ 1+\zeta+\zeta^8,\ 1+ \zeta^2 + \zeta^7,\ 1 + \zeta^4 + \zeta^5\ \}&.\end{align*}
Using $\zeta^2,\zeta^3,\zeta^4,\zeta^5,\zeta^6,\zeta^7$ as a $\mathbb{Z}$-basis of $\mathbb{Z}[\zeta]$ (cf.\ \cite[Chapter 1, (10.2) Proposition]{Neukirch_english}) this gives
\[(\ -\varepsilon_{\gapcom{9b}}(u) + \varepsilon_{\gapcom{9c}}(u),\  \varepsilon_{\gapcom{9a}}(u) - \varepsilon_{\gapcom{9b}}(u) \ ) \in \{\ (-1,-1), \ (1,0), \ (0,1)\ \}.\]
Combining each of this possibilities with $\varepsilon_{9a}(u) + \varepsilon_{9b}(u) + \varepsilon_{9c}(u) = 1$, we get
\[(\ \varepsilon_{\gapcom{9a}}(u),\ \varepsilon_{\gapcom{9b}}(u),\ \varepsilon_{\gapcom{9c}}(u)\ ) \in \{\ (1,0,0),\ (0,1,0),\ (0,0,1)\ \}.\]
Thus $u$ is rationally conjugate to a group element. This is also a consequence of \cite[Proposition 1]{SylowLeo}.
 
Now assume $u$ is of order 12, then $G$ contains an element of order 12 by Proposition \ref{Hertweck}. So let $\gapcom{2a}, \gapcom{3a}, \gapcom{4a}, \gapcom{6a}, \gapcom{12a}, \gapcom{12b}$ be the conjugacy classes with potentially non-vanishing partial augmentations for $u$. Let $\zeta$ be a primitive 12-th root of unity such that $\Theta_1(\gapcom{12a}) \sim (1,\zeta,\zeta^{11}).$ 

We will use $\zeta, \zeta^4, \zeta^8, \zeta^{11}$ as a $\mathbb{Z}$-basis of $\mathbb{Z}[\zeta]$. (This is a basis since $\varphi(12) = 4$, where $\varphi$ denotes Euler's totient function, and $1 = -\zeta^4 - \zeta^8, \, \ \zeta^2 = -\zeta^8, \zeta^3 = \zeta - \zeta^{11}, \ \zeta^5 = -\zeta^{11}, \  \zeta^6 = -1 = \zeta^4 + \zeta^8, \ \zeta^7 = -\zeta, \ \zeta^9 = -\zeta + \zeta^{11}$ and $\zeta^{10} = -\zeta^4$.) We have 
\begin{eqnarray}
\label{1}
\varepsilon_{\gapcom{2a}}(u) + \varepsilon_{\gapcom{3a}}(u) +\varepsilon_{\gapcom{4a}}(u) +\varepsilon_{\gapcom{6a}}(u) +\varepsilon_{\gapcom{12}a}(u) + \varepsilon_{\gapcom{12}b}(u) = 1.
\end{eqnarray}

\begin{table}[h]\centering
\begin{tabular}{l|c c c c c c c}
{ } & \gapcom{1a} & \gapcom{2a} & \gapcom{3a} & \gapcom{4a} & \gapcom{6a} & \gapcom{12a} & \gapcom{12b} \\ \hline
$\varphi_{1}$ & 3 & $-1$ & $\cdot$ & 1 & 2 & $1+\zeta_{12}+\zeta_{12}^{-1}$ & $1-\zeta_{12}-\zeta_{12}^{-1}$  \\
$\varphi_{2}$ & 5 & 1 & $-1$ & $-1$ & 1 & $2+\zeta_{12}+\zeta_{12}^{-1}$ & $2-\zeta_{12}-\zeta_{12}^{-1}$ \\
$\varphi_{3}$ & 7 & $-1$ & 1 & $-1$ & $-1$ & $2+\zeta_{12}+\zeta_{12}^{-1}$ & $2-\zeta_{12}-\zeta_{12}^{-1}$ \\
$\varphi_{5}$ & 11 & $-1$ & $-1$ & 1 & $-1$ & 1 & 1\\ 
\end{tabular}\\[.4cm]
 \caption{Parts of some $p$-Brauer-characters of $G = \PSL(2,p^f)$ for $p \notin \{2,3\}$ and $24 \mid |G|$.}\label{bratabPSL_2_23_23}
 \end{table}

 Furthermore, $\Theta_1(u^9) \sim (1, \zeta^3, \zeta^9)$ and $\Theta_1(u^4) \sim (1, \zeta^4, \zeta^8).$ Thus, as $\varphi_1$ has only real values, $\Theta_1(u) \sim X$ with $X \in \{(1,\zeta^5,\zeta^7),(1,\zeta,\zeta^{11})\} = \{(1,-\zeta^{11},-\zeta),(1,\zeta,\zeta^{11})\}.$ Hence, using Table \ref{bratabPSL_2_23_23}, we obtain $-\varepsilon_{\gapcom{2a}}(u) + \varepsilon_{\gapcom{4a}}(u) + 2\varepsilon_{\gapcom{6a}}(u) + (1+\zeta + \zeta^{11})\varepsilon_{\gapcom{12a}}(u) + (1-\zeta-\zeta^{11})\varepsilon_{\gapcom{12b}}(u) \in \{ 1 + \zeta + \zeta^{11}, 1 - \zeta - \zeta^{11}\}.$  Using again $\zeta, \ \zeta^4,\ \zeta^8, \ \zeta^{11}$ as a basis of $\mathbb{Z}[\zeta]$ this gives 
\begin{align}
\label{2}
\varepsilon_{\gapcom{12a}}(u)-\varepsilon_{\gapcom{12b}}(u) &= \pm 1, \\
\label{3}
-\varepsilon_{\gapcom{2a}}(u) + \varepsilon_{\gapcom{4a}}(u) + 2\varepsilon_{\gapcom{6a}}(u) + \varepsilon_{\gapcom{12a}}(u) + \varepsilon_{\gapcom{12b}}(u) &= 1. 
\end{align} 

Proceeding the same way we have $\Theta_2(u^9) \sim (1,-1,-1, \zeta^3,\zeta^9)$, $\Theta_2(u^4) \sim (1,\zeta^4,\zeta^8,\zeta^4,\zeta^8)$ and $\Theta_2(u) \sim X$ with $X \in \{(1, \zeta^2, \zeta^{10}, \zeta, \zeta^{11}), (1,\zeta^2, \zeta^{10}, \zeta^5, \zeta^{7})\} .$ So, by Table \ref{bratabPSL_2_23_23} and $\zeta^2+\zeta^{10}=1$, we get $\varepsilon_{\gapcom{2a}}(u) - \varepsilon_{\gapcom{3a}}(u) - \varepsilon_{\gapcom{4a}}(u) + \varepsilon_{\gapcom{6a}}(u) + (2+\zeta+\zeta^{11}) \varepsilon_{\gapcom{12a}}(u) + (2-\zeta-\zeta^{11}) \varepsilon_{\gapcom{12b}}(u) \in \{2 + \zeta + \zeta^{11}, 2 - \zeta - \zeta^{11}\}$. Comparing coefficients of $\zeta^4$ gives
\begin{eqnarray}
\label{4}
\varepsilon_{\gapcom{2a}}(u) - \varepsilon_{\gapcom{3a}}(u) - \varepsilon_{\gapcom{4a}}(u) + \varepsilon_{\gapcom{6a}}(u) + 2\varepsilon_{\gapcom{12a}}(u) + 2\varepsilon_{\gapcom{12b}}(u) = 2.
\end{eqnarray}

Applying the same for $\varphi_3$ we obtain $\Theta_3(u^9) \sim (1,-1,-1, \zeta^3, \zeta^9, \zeta^3,\zeta^9)$, $\Theta_3(u^4) \sim (1,\zeta^4,\zeta^8,1,\zeta^4,\zeta^8,1)$ and $\Theta_3(u) \sim (X)$ with \begin{align*} X \in \{ & (1,-1,-1, \zeta, \zeta^{11}, \zeta, \zeta^{11}), (1,-1,-1,\zeta,\zeta^{11},\zeta^5,\zeta^7),  (1,-1,-1,\zeta^5,\zeta^7,\zeta^5,\zeta^7),\\ & (1,\zeta^2, \zeta^{10}, \zeta^3, \zeta^{9}, \zeta, \zeta^{11}), (1,\zeta^2, \zeta^{10}, \zeta^3, \zeta^{9}, \zeta^5, \zeta^{7})\}.\end{align*} So, by Table \ref{bratabPSL_2_23_23}, $\zeta^2+\zeta^{10}=1$ and $\zeta^3+\zeta^9=0$, we get \begin{align*} \varepsilon_{\gapcom{2a}}(u) + \varepsilon_{\gapcom{3a}}(u) - \varepsilon_{\gapcom{4a}}(u) - \varepsilon_{\gapcom{6a}}(u) + (2+\zeta+\zeta^{11}) \varepsilon_{\gapcom{12a}}(u) + (2-\zeta-\zeta^{11}) \varepsilon_{\gapcom{12b}}(u) &\\ \in \{-1 + 2\zeta + 2\zeta^{11},\ -1,\ -1 -2\zeta - 2\zeta^{11},\ 2+\zeta+\zeta^{11},\ 2 - \zeta - \zeta^{11}\}&. \end{align*} As the first three possibilities would give $\varepsilon_{\gapcom{12a}}(u) - \varepsilon_{\gapcom{12b}}(u) \in \{-2,0,2\}$, contradicting \eqref{2}, only the last two remain and give
\begin{eqnarray}
\label{5}
-\varepsilon_{\gapcom{2a}}(u) + \varepsilon_{\gapcom{3a}}(u) - \varepsilon_{\gapcom{4a}}(u) - \varepsilon_{\gapcom{6a}}(u) + 2\varepsilon_{\gapcom{12a}}(u) + 2\varepsilon_{\gapcom{12b}}(u) = 2.
\end{eqnarray}

In the same way %$\Theta_{5}(u^6) \sim (1,1,1,1,1,-1,-1,-1,-1,-1,-1)$, 
$\Theta_{5}(u^6) \sim ( \eigbox{5}{1}, \eigbox{6}{-1})$,
%$\Theta_{5}(u^9) \sim (1,1,1,-1,-1,\zeta^3,\zeta^9,\zeta^3,\zeta^9,\zeta^3,\zeta^9)$,
$\Theta_{5}(u^9) \sim ( \eigbox{3}{1}, \eigbox{2}{-1}, \eigbox{3}{\zeta^3,\zeta^9,})$,
%$\Theta_{5}(u^4) \sim (1, \zeta^4, \zeta^8, 1, \zeta^4, \zeta^8, 1, \zeta^4, \zeta^8,\zeta^4, \zeta^8)$,
$\Theta_{5}(u^4) \sim ( \eigbox{3}{1}, \eigbox{4}{\zeta^4, \zeta^8})$
 and  $\Theta_{5}(u) \sim (1, \zeta, \zeta^2, \zeta^3, \zeta^4, \zeta^5, \zeta^7, \zeta^8,\zeta^9, \zeta^{10},\zeta^{11})$ (note that $\varphi_{5}(u)$ must not only be real, but even rational, as $\varphi_{5}$ has only rational values). Thus $-\varepsilon_{\gapcom{2a}}(u) - \varepsilon_{\gapcom{3a}}(u) + \varepsilon_{\gapcom{4a}}(u) - \varepsilon_{\gapcom{6a}}(u) +  \varepsilon_{\gapcom{12a}}(u) +  \varepsilon_{\gapcom{12b}}(u) = 1$ giving 
 \begin{eqnarray}
\label{6} 
  -\varepsilon_{\gapcom{2a}}(u) - \varepsilon_{\gapcom{3a}}(u) + \varepsilon_{\gapcom{4a}}(u) - \varepsilon_{\gapcom{6a}}(u) + \varepsilon_{\gapcom{12a}}(u) + \varepsilon_{\gapcom{12b}}(u) = 1.
 \end{eqnarray}
 
Now subtracting \eqref{1} from \eqref{6} gives $\varepsilon_{\gapcom{2a}}(u)+\varepsilon_{\gapcom{3a}}(u)+\varepsilon_{\gapcom{6a}}(u)=0$ while subtracting \eqref{4} from \eqref{5} gives $\varepsilon_{\gapcom{2a}}(u)-\varepsilon_{\gapcom{3a}}(u)+\varepsilon_{\gapcom{6a}}(u)=0$. Thus $\varepsilon_{\gapcom{3a}}(u)=0$. Then subtracting \eqref{1} from \eqref{3} gives $-2\varepsilon_{\gapcom{2a}}(u)+\varepsilon_{\gapcom{6a}}(u)=0$, so $\varepsilon_{\gapcom{2a}}(u)=\varepsilon_{\gapcom{6a}}(u)=0.$ Now multiplying \eqref{1} by 2 and subtracting it from \eqref{4} gives $\varepsilon_{\gapcom{4a}}(u)=0.$ Using \eqref{1} and \eqref{2} this leaves only the trivial possibilities $(\varepsilon_{\gapcom{2a}}(u),\varepsilon_{\gapcom{3a}}(u),\varepsilon_{\gapcom{4a}}(u),\varepsilon_{\gapcom{6a}}(u),\varepsilon_{\gapcom{12a}}(u),\varepsilon_{\gapcom{12b}}(u)) \in \{(0,0,0,0,1,0),(0,0,0,0,0,1)\}$.

For part b) assume $p \not\in \{2, 5\}$, $u \in \V(\mathbb{Z}G)$ is of order 10 and let $\zeta$ be a primitive $5$-th root of unity s.t. $\Theta_1(\gapcom{10a}) \sim (1,-\zeta,-\zeta^{4})$. Assume further that $u^2$ is rationally conjugate to an element in $\gapcom{5a}$. 

\begin{table}[h]\centering
\begin{tabular}{l|c c c c c c}
{ } & \gapcom{1a} & \gapcom{2a} & \gapcom{5a} & \gapcom{5b} & \gapcom{10a} & \gapcom{10b} \\ \hline
$\varphi_{1}$ & 3 & $-1$ & $-\beta$ & $-\alpha$  &  $-2\alpha -\beta$ & $-\alpha -2\beta$  \\
$\varphi_{2}$  & 5 & 1 & $\cdot$ & $\cdot$ & $-2\alpha$ & $ -2\beta$ \\
\end{tabular} \\[.4cm]

\begin{tabular}{rll} with & $\alpha = \zeta_5 + \zeta_5^4$, &  $\beta = \zeta_5^2 + \zeta_5^3$ \end{tabular}
\caption{Part of some $p$-Brauer-characters of $G = \PSL(2,p^f)$ with $p \notin \{2,5\}$ and $20 \mid |G|$}\label{bratabPSL_2_19_19}
\end{table}
We have 
\[\varepsilon_{\gapcom{2a}}(u) + \varepsilon_{\gapcom{5a}}(u) +\varepsilon_{\gapcom{5b}}(u) +\varepsilon_{\gapcom{10a}}(u) +\varepsilon_{\gapcom{10b}}(u) = 1.\]
Furthermore, $\Theta_1(u^5) \sim(1,-1,-1)$ and $\Theta_1(u^6) \sim (1,\zeta^2,\zeta^3).$ As $\varphi_1$ has only real values, we get $\Theta_1(u) \sim (1,-\zeta^2, -\zeta^3).$ Thus  
\begin{align*}
 - \varepsilon_{\gapcom{2a}}(u) \ &+ \ (-\zeta^2 - \zeta^3) \varepsilon_{\gapcom{5a}}(u) \ + \ (-2\zeta - \zeta^2 - \zeta^3 - 2\zeta^4) \varepsilon_{\gapcom{10a}}(u)  \\
  &+ \ (-\zeta - \zeta^4) \varepsilon_{\gapcom{5b}}(u) \  + \ (-\zeta - 2\zeta^2 - 2\zeta^3 - \zeta^4) \varepsilon_{\gapcom{10b}}(u)  \ = \ 1 - \zeta^2 - \zeta^3.
\end{align*} Using $\zeta,\ \zeta^2,\ \zeta^3,\ \zeta^4$ as a $\mathbb{Z}$-basis of $\mathbb{Z}[\zeta]$ we obtain 
\begin{align*}
\varepsilon_{\gapcom{2a}}(u) - \varepsilon_{\gapcom{5b}}(u) -2\varepsilon_{\gapcom{10a}}(u) - \varepsilon_{\gapcom{10b}}(u) & = -1, \\ \varepsilon_{\gapcom{2a}}(u) - \varepsilon_{\gapcom{5a}}(u) -\varepsilon_{\gapcom{10a}}(u) -2 \varepsilon_{\gapcom{10b}}(u) &= -2.
\end{align*}
 In the same way we get $\Theta_2(u^5) \sim (1,1,1,-1,-1)$, $\Theta_2(u^6) \sim (1,\zeta,\zeta^2,\zeta^3,\zeta^4)$ and $\Theta_2(u) \sim X$ with $X \in \{(1,-\zeta,\zeta^2, \zeta^3, - \zeta^4), (1,\zeta,-\zeta^2,-\zeta^3,\zeta^4)\}.$ We have $\varphi_2(u) = \varepsilon_{\gapcom{2a}}(u) -2(\zeta + \zeta^4)\varepsilon_{\gapcom{10a}}(u) - 2(\zeta^2 +  \zeta^3) \varepsilon_{\gapcom{10b}}(u).$ Hence 
 \[(\ -\varepsilon_{\gapcom{2a}}(u)-2\varepsilon_{\gapcom{10a}}(u),\ -\varepsilon_{\gapcom{2a}}(u) - 2\varepsilon_{\gapcom{10b}}(u)\ ) \in \{\ (-2,0),\ (0,-2)\ \}.\] 
Combining these equations with the equations obtained above we get \[ (\varepsilon_{\gapcom{2a}}(u), \varepsilon_{\gapcom{5a}}(u),\varepsilon_{\gapcom{5b}}(u),\varepsilon_{\gapcom{10a}}(u),\varepsilon_{\gapcom{10b}}(u)) \in \{\ (0,1,-1,1,0),\ (0,0,0,0,1)\ \}.\]
If $u^2$ is rationally conjugate to an element in $\gapcom{5b}$, then replacing every $\zeta^i$ with $\zeta^{2i}$ and doing the same computations as above gives the result. \end{proof}

\begin{proof}[Proof of Theorem 1] By Proposition \ref{Hertweck}, to obtain the Zassenhaus Conjecture for $G = \PSL(2,23)$ only elements of order $4$ and $12$ in $\V(\mathbb{Z}G)$ need to be checked and these are rationally conjugate to group elements by Lemma \ref{HeLPPSL}. So from now on let $G = \PSL(2,19)$. Then by Proposition \ref{Hertweck} only elements of order 9 and 10 need to be checked, but elements of order 9 are already handled in Lemma \ref{HeLPPSL}. So assume $u \in \V(\mathbb{Z}G)$ is of order 10 and not rationally conjugate to a group element. If $u$ is not rationally conjugate to an element of $G$, then also $u^3$ is not rationally conjugate to an element in $G$. Furthermore, if $u^2$ is rationally conjugate to an element in $\gapcom{5a}$, then $u^6$ is rationally conjugate to an element in $\gapcom{5b}$. So we may assume that $u^2$ is conjugate to an element in $\gapcom{5a}$ and by Lemma \ref{HeLPPSL} we get 
\[(\varepsilon_{\gapcom{2a}}(u), \varepsilon_{\gapcom{5a}}(u),\varepsilon_{\gapcom{5b}}(u),\varepsilon_{\gapcom{10a}}(u),\varepsilon_{\gapcom{10b}}(u)) = (0,1,-1,1,0).\]

We give the parts of the character tables relevant for our proof in the Tables \ref{chartabPSL_2_19}  and \ref{bratabPSL_2_19_5}.

\begin{table}[h]\centering
\begin{tabular}{l|c c c c c c}
{ } & \gapcom{1a} & \gapcom{2a} & \gapcom{5a} & \gapcom{5b} & \gapcom{10a} & \gapcom{10b} \\ \hline
$\chi_{\gapcom{18}}$ & 18 & $-2$ & $-\alpha$ & $-\beta$ & $-\alpha$ & $-\beta$ \\
$\chi_{\gapcom{19}}$ & 19 & $-1$ & $-1$ & $-1$ & $-1$ & $-1$ \\
\end{tabular} \\[.4cm]

\begin{tabular}{rll} with & $\alpha = \zeta_5 + \zeta_5^4$, &  $\beta = \zeta_5^2 + \zeta_5^3$   \end{tabular}
\caption{Part of the ordinary character table of $\PSL(2,19)$}\label{chartabPSL_2_19}
\end{table}

\begin{table}[h]\centering
\subfloat[Part of the Brauer table]{
\begin{tabular}{l|c c}
{ } & \gapcom{1a} & \gapcom{2a} \\ \hline
$\varphi_{\gapcom{1}}$ & 1 & 1  \\
$\varphi_{\gapcom{18}}$  & 18 & -2 \\
\end{tabular}}
\qquad
\subfloat[Part of the decomposition matrix]{
\begin{tabular}{l|c c }
{ } & $\varphi_{\gapcom{1}}$  & $\varphi_{\gapcom{18}}$ \\ \hline
$\chi_{\gapcom{1}}$ & 1 & $\cdot$ \\
$\chi_{\gapcom{18}}$ & $\cdot$ & 1  \\
$\chi_{\gapcom{19}}$ & 1 & 1\end{tabular}}
\caption{Part of Brauer table and decomposition matrix of $\PSL(2,19)$ for the prime $5$}\label{bratabPSL_2_19_5}
\end{table}

Let $D_{\gapcom{18}}$ and $D_{\gapcom{19}}$ be representations affording the characters $\chi_{\gapcom{18}}$ and $\chi_{\gapcom{19}}$ given in Table \ref{chartabPSL_2_19}. We compute the eigenvalues of $D_{\gapcom{18}}(u)$ and $D_{\gapcom{19}}(u)$ using the character table in the way demostrated above: We have $D_{\gapcom{18}}(u^5) \sim (\eigbox{8}{1}, \eigbox{10}{-1})$ and $D_{\gapcom{18}}(u^6) \sim (\eigbox{3}{1, \zeta, \zeta^2, \zeta^3, \zeta^4}, \eigbox{1}{1, \zeta, \zeta^4})$. Since $\chi_{\gapcom{18}}(u) = \chi_{\gapcom{18}}(\gapcom{5a}) - \chi_{\gapcom{18}}(\gapcom{5b}) + \chi_{\gapcom{18}}(\gapcom{10a}) = -2(\zeta + \zeta^4) + (\zeta^2 + \zeta^3)$ we obtain
\[D_{\gapcom{18}}(u) \sim (1,\zeta,\zeta^2,\zeta^3,\zeta^4,1,\zeta^2,\zeta^3,-1,-\zeta,-\zeta^2,-\zeta^3,-\zeta^4,-1,-\zeta,-\zeta^4,-\zeta,-\zeta^4).\]
Moreover $D_{\gapcom{19}}(u^5) \sim (\eigbox{9}{1}, \eigbox{10}{-1})$ and $D_{\gapcom{19}}(u^6) \sim (\eigbox{3}{1}, \eigbox{4}{\zeta, \zeta^2, \zeta^3, \zeta^4})$. Since $\chi_{\gapcom{19}}(u) = -1$ we get
\[D_{\gapcom{19}}(u) \sim (1,\zeta,\zeta^2,\zeta^3,\zeta^4,\zeta,\zeta^2,\zeta^3,\zeta^4,-1,-\zeta,-\zeta^2,-\zeta^3,-\zeta^4,-1,-\zeta,-\zeta^2,-\zeta^3,-\zeta^4).\]

By \cite{SLSchur} the Schur indices of all irreducible representations of $G$ are 1 and thus we may assume that $D_{\gapcom{19}}$ is a $\mathbb{Q}_5$-representation while $D_{\gapcom{18}}$ is a $K$-representation, where $K$ is the $5$-adic completion of $\mathbb{Q}(\zeta_9 + \zeta_9^{-1}, \zeta_5 + \zeta_5^{-1}) = \mathbb{Q}(\zeta_9 + \zeta_9^{-1}, \sqrt{5}).$ For both these fields, the rings of integers are principal ideal domains, so by \cite[Proposition 23.16]{CurtisReiner1} we may assume that $D_{\gapcom{19}}$ is a $\mathbb{Z}_5$-representation and $D_{\gapcom{18}}$ is a $R$-representation, where $R$ denotes the ring of integers of $K.$ Let $L_{\gapcom{19}}$ and $L_{\gapcom{18}}$ be a $\mathbb{Z}_5G$-lattice and an $RG$-lattice respectively affording these representations. As usual, denote by $\bar{\phantom{x}}$ the reduction modulo the maximal ideal of $\mathbb{Z}_5$ and that one of $R$. Denote by $k$ a field of characteristic $5$ which contains $\bar{\mathbb{Z}}_5$ and $\bar{R}$ and affords all irreducible $5$-modular representations of $G$. 

%As $D_{19}$ is a deleted permutation representation (i.e. the module corresponding to the representation is isomorphic to a permutation module factored by the trivial submodule) coming from the action of $\PSL(2,19)$ on the projective line over $\mathbb{F}_{19}$, we may assume that $D_{19}$ is a $\mathbb{Z}$-representation, so also a $\mathbb{Z}_5$-representation. By a theorem of Fong \cite[Corollary 10.13]{Isaacs} we may assume that $D_{18}$ is a $K$-representation, where $K$ is an unramified extension of $\mathbb{Q}_5(\zeta + \zeta^4)$. Denote by $R$ the ring of integers of $K.$ As usual denote by $\bar{}$\ the reduction modulo the maximal ideal of $\mathbb{Z}_5$ and that one of $R$. We may view $\bar{\mathbb{Z}}_5$ as a subfield of $\bar{R} =: k$.

We may assume that $\bar{L}_{\gapcom{19}}$ contains $\bar{L}_{\gapcom{18}}$ as submodule (multiplying a module by the augmentation ideal $I(kG)$ annihilates precisely the trivial $kG$-submodules).
%Let $I(kG)$ be the augmentation ideal of $kG$. Then $I(kG)\bar{L}_{19} \cong \bar{L}_{18}$, as $I(kG)$ annihilates exactly the trivial $kG$-modules. So $\bar{L}_{18}$ is a submodule of $\bar{L}_{19}$ and by the Brauer table given above
$\bar{L}_{\gapcom{19}}/\bar{L}_{\gapcom{18}}$ is a trivial $kG$-module, so also a trivial $k\langle \bar{u} \rangle$-module. By Proposition \ref{prop_lattices_cyclic_group}, as $\mathbb{Z}_5\langle u \rangle$-lattice and as $R\langle u \rangle$-lattice respectively, we may write $L_{\gapcom{19}} \cong L^{+}_{\gapcom{19}} \oplus L^{-}_{\gapcom{19}}$ and $L_{\gapcom{18}} \cong L^+_{\gapcom{18}} \oplus L^{-}_{\gapcom{18}}$ respectively, such that the composition factors of $\bar{L}^+_{\gapcom{i}}$ are all trivial and the composition factors of $\bar{L}^{-}_{\gapcom{i}}$ are all non-trivial as $k\langle \bar{u} \rangle$-modules for $\gapcom{i} \in \{\gapcom{18},\gapcom{19}\}$. As $\bar{L}_{\gapcom{19}}/\bar{L}_{\gapcom{18}}$ is a trivial module, we have $\bar{L}^{-}_{\gapcom{18}} \cong \bar{L}^{-}_{\gapcom{19}}$ (as $k\langle \bar{u} \rangle$-modules).

By the computations above the eigenvalues of $D_{\gapcom{19}}(u)$, which are not $5$-th roots of unity, i.e.\ which contribute to $L^{-}_{\gapcom{19}}$ by Proposition \ref{prop_lattices_cyclic_group}, are
\[ (-1,-\zeta,-\zeta^2,-\zeta^3,-\zeta^4,-1,-\zeta,-\zeta^2,-\zeta^3,-\zeta^4).\]
Recall that we denote by $(\mathbb{Z}_5)^-, I(\mathbb{Z}_5C_5)^-$ and $(\mathbb{Z}_5C_5)^-$ the indecomposable $\mathbb{Z}_5C_{10}$-lattices of rank $1$, $4$ and $5$ respectively, which have non-trivial composition factors, cf.\ Proposition \ref{prop_lattices_cyclic_group} and Proposition \ref{prop_lattices_cyclic_group_of_prime_order}. By Proposition \ref{prop_lattices_cyclic_group_of_prime_order} the eigenvalues imply $L^{-}_{\gapcom{19}} \cong X$ with
\[X \in \{\ {2(\mathbb{Z}_5)^- \oplus 2I(\mathbb{Z}_5C_5)^-}, \  {(\mathbb{Z}_5)^{-} \oplus I(\mathbb{Z}_5C_5)^{-} \oplus (\mathbb{Z}_5C_5)^{-}}, \ {2(\mathbb{Z}_5C_5)^{-}}\ \}.\]
In any case $\bar{L}^{-}_{\gapcom{19}}$ has two indecomposable summands of $k$-dimension at least $4$, as indecomposable summands of $X$ stay indecomposable after reduction by Proposition 
\ref{prop_lattices_cyclic_group_of_prime_order}.

On the other hand the eigenvalues of $D_{\gapcom{18}}(u)$, which are not 5-th roots of unity are 
\[(-1,-\zeta,-\zeta^2,-\zeta^3,-\zeta^4,-1,-\zeta,-\zeta^4,-\zeta,-\zeta^4).\]
Note that the simple $R\langle u \rangle$-lattice $S$ affording the eigenvalues $(-\zeta^2, -\zeta^3)$ appears exactly once as a composition factor of $L_{\gapcom{18}}^{-}$. Let $L^{-}_{\gapcom{18}} \cong Y \oplus Z$ such that $Y$ is indecomposable and $S$ is a composition factor of $Y$. There are at most two non-isomorphic simple $R\langle u \rangle$-lattices involved in $Z$, namely the one affording eigenvalues ($-\zeta, -\zeta^4)$ and the one affording the eigenvalue $-1$. Hence by Proposition \ref{num_of_components}, the maximal $R$-rank of an indecomposable summand of $Z$ is 3. Again by Proposition \ref{num_of_components} both simple lattices corresponding to the eigenvalues $(-\zeta, -\zeta^{4})$ and $(-\zeta^2,-\zeta^{3})$, which both have $R$-rank 2, appear each at most once as a composition factor of $Y$, while the simple lattice corresponding to the eigenvalue $-1$ appears at most twice. Thus the maximal $R$-rank of $Y$ is 6. So in any case $\bar{L}_{\gapcom{18}}^{-}$ does not posses two indecomposable direct summands of dimension at least 4. But since the Krull-Schmidt-Azumaya Theorem holds, we obtain a contradiction to $\bar{L}^{-}_{\gapcom{18}} \cong \bar{L}^{-}_{\gapcom{19}}$ and the above paragraph. 
%\begin{rem} The Schur index of every irreducible character of $\PSL(2,q)$ is actually 1 \cite{SLSchur}, so in the above proof we may assume $K = \mathbb{Q}_5(\zeta_5 + \zeta_5^4)$ and so we may avoid the use of the theorem of Fong. However, Fong's theorem could be really helpful, if this method is used for other groups, as demonstrated above. \end{rem}
\end{proof}  

\begin{rem}
Naturally the next candidate to study the Zassenhaus Conjecture for $\PSL(2,p)$ would be $\PSL(2,29)$. After applying the HeLP-method for this group, units of order 14 remain critical. Using the method from above would in any case involve $RC_7$-lattices, where $R$ denotes the ring of integers of $\mathbb{Q}_7(\zeta_7 + \zeta_7^{-1}).$ Since the representation type of $RC_7$ is wild, see \cite{Dieterich}, this seems hopeless using only theoretical arguments. However the degree of the representations involved is at most 30, so a computational approach does not seem out of reach. This is however not part of this paper.
\end{rem}

\subsection{Proof of Theorem 2.}

Let $G = \Aut(A_6)$, the automorphism group of the alternating group of degree $6$. Both, $M_{10}$ and $\PGL(2,9)$, are subgroups of index $2$ in $G$ (cf.\ Figure \ref{normal_subgroups_AutA6}). There is a unique conjugacy class \gapcom{3a} of elements of order 3 in $G$ (it is the union of the two conjugacy classes of elements of order $3$ in $A_6$ and has length $80$). This class is also the conjugacy class of elements of order $3$ in $M_{10}$ and $\PGL(2,9)$. Furthermore there is a conjugacy class of $G$ consisting of all the involutions in $A_6$ of length $45$, which we will denote by \gapcom{2a}. This is clearly also a conjugacy class of $M_{10}$ and $\PGL(2,9)$. Let $u$ be a unit of order 6 in $\V(\mathbb{Z}M_{10})$ or in $\V(\mathbb{Z}\PGL(2,9))$. By \cite[Proof of Theorem 3.1]{KiKo}, if such a unit exists, then all partial augmentations of $u$ vanish except at the classes \gapcom{2a} and \gapcom{3a} and then $(\varepsilon_{\gapcom{2a}}(u),\varepsilon_{\gapcom{3a}}(u)) = (-2,3)$. We will first compute in $\Aut(A_6)$ and then obtain contradictions to the existence of $u$ via restriction to $M_{10}$ and $\PGL(2,9)$.

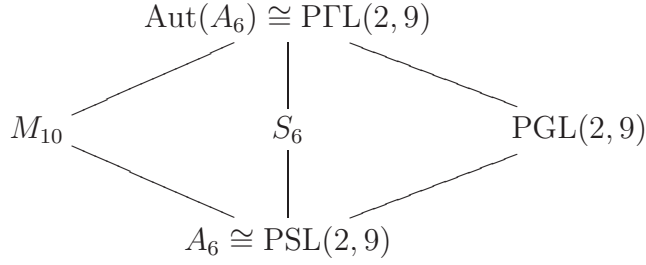
\begin{figure}[h]
\begin{center}
\begin{equation*}
  \xymatrix{
  & \Aut(A_6) \cong {\text{P} \Gamma \text{L}}(2,9) \ar@{-}[dl] \ar@{-}[d] \ar@{-}[dr] & \\
  M_{10} \ar@{-}[dr] & S_6 \ar@{-}[d] & \PGL(2,9) \ar@{-}[dl] \\
  & A_6 \cong \PSL(2,9) & }
\end{equation*}
\caption{Almost simple groups containing $A_6$. Indices of consecutive subgroups are 2.}
\label{normal_subgroups_AutA6}
\end{center}
\end{figure}

% By \cite{KiKoStAndrews} (also in \cite{KiKo}) % KiKoStAndrews removed -- I think the possible part augs are not given there
% By \cite[Proof of Theorem 3.1]{KiKo}, if a unit $u$ of order 6 exists in $\V(\mathbb{Z}M_{10})$ or in $\V(\mathbb{Z}\PGL(2,9))$, then all partial augmentations of $u$ vanish except at the classes of elements of order 2 and 3 in $A_6$ (note that $A_6$ is a normal subgroup of index 2 in $M_{10}$ and in $\PGL(2,9)$) and then $(\varepsilon_{2a}(u),\varepsilon_{3a}(u)) = (-2,3)$. Both groups, $M_{10}$ and $\PGL(2,9)$, are normal subgroups of $\Aut(A_6)$ and the conjugacy classes $2a$ and $3a$ are the same in all these groups. We will first compute in $\Aut(A_6)$ and then obtain contradictions to the existence of $u$ via restriction to $M_{10}$ and $\PGL(2,9)$, (cf.\ Figure \ref{normal_subgroups_AutA6}).

The relevant parts of the character tables\footnote{These tables can be obtained in \gapcom{GAP} by calling \gapcom{CharacterTable("A6.2\^{}2");} and \gapcom{CharacterTable("A6.2\^{}2") mod 3;} respectively.} of $G$ are given in Table \ref{chartabA6}, the corresponding decomposition matrix in Table \ref{decompmatA6}. By $\chi_{\gapcom{10}}$ we denote the irreducible character of degree $10$ which contains a trivial constituent after reduction modulo $3$.

\begin{table}[h]
\centering
\subfloat[Used part of the ordinary character table]{
\begin{tabular}{l|c c c}
{ } & \gapcom{1a} & \gapcom{2a} & \gapcom{3a} \\ \hline
$\chi_{\gapcom{1a}}$ & 1 & 1 & 1 \\
$\chi_{\gapcom{1b}}$ & 1 & 1 & 1 \\
$\chi_{\gapcom{10}}$ & 10 & 2 & 1  \\
$\chi_{\gapcom{20}}$ & 20 & $-4$ & 2 \\
\end{tabular}}
\qquad
\subfloat[Used part of the Brauer table for $p=3$]{
\begin{tabular}{l|c c c c}
{ } & \gapcom{1a} & \gapcom{2a} & \gapcom{5a} & \gapcom{2b} \\ \hline
$\varphi_{\gapcom{1a}}$ & $1$ & $1$ & $1$ & $1$  \\
$\varphi_{\gapcom{1b}}$ & $1$ & $1$ & $1$ & $-1$ \\
$\varphi_{\gapcom{6a}}$ & $6$ & $-2$ &	 $1$ & . \\
$\varphi_{\gapcom{6b}}$ & $6$ & $-2$ & $1$ & $.$ \\
$\varphi_{\gapcom{8}}$ & $8$ & $.$ & $-2$ & $.$ \\
\end{tabular}}
\caption{Parts of the ordinary character table and the Brauer table for the prime $3$ for the group $\Aut(A_6)$}\label{chartabA6}
\end{table}

\begin{table}[h]\centering
\begin{tabular}{l|c c c c c}
{ } & $\varphi_{\gapcom{1a}}$  & $\varphi_{\gapcom{1b}}$ & $\varphi_{\gapcom{6a}}$ & $\varphi_{\gapcom{6b}}$ & $\varphi_{\gapcom{8}}$ \\ \hline
$\chi_{\gapcom{1a}}$ & $1$ & $\cdot$ & $\cdot$ & $\cdot$ & $\cdot$ \\
$\chi_{\gapcom{1b}}$ & $\cdot$ & $1$ & $\cdot$ & $\cdot$ & $\cdot$ \\
$\chi_{\gapcom{10}}$ & $1$ & $1$ & $\cdot$ & $\cdot$ & $1$  \\
$\chi_{\gapcom{20}}$ & $\cdot$ & $\cdot$ & $1$ & $1$ & $1$ \\
\end{tabular}
\caption{Part of the decomposition matrix of $\Aut(A_6)$ for the prime $3$}\label{decompmatA6}
\end{table}

% Set $G = \Aut(A_6)$ and let $u$ be a unit of order 6 in $\V(\mathbb{Z}G)$ such that $(\varepsilon_{2a}(u),\varepsilon_{3a}(u))= (-2,3)$ and $\varepsilon_g(u)=0$ for all other conjugacy classes  ($2a$ is the $\Aut(A_6)$-conjugacy class of involutions of length $45$ and $3a$ is the unique $\Aut(A_6)$-conjugacy class of elements of order $3$). 
Denote by $\zeta$ a complex primitive $3$-rd root of unity. Using the HeLP-method and the fact that each $\chi_{\gapcom{i}}$ is real valued, we obtain
\begin{align*} D_{\gapcom{10}}(u) &\sim (\eigbox{2}{1}, \eigbox{2}{\zeta, \zeta^2}, \eigbox{2}{-1}, \eigbox{1}{-\zeta, -\zeta^2} ),  \\ D_{\gapcom{20}}(u) &\sim (\eigbox{8}{1}, \eigbox{6}{-\zeta, -\zeta^2} ). \end{align*}
% \begin{gather*}
% D_{10}(u) \sim (1,\zeta , \zeta^2, 1, \zeta, \zeta^2 , -1, -\zeta, -\zeta^2, -1)  \\
% D_{20}(u) \sim (1,1,1,1,1,1,1,1,-\zeta , -\zeta^2 ,-\zeta , -\zeta^2 ,-\zeta , -\zeta^2 ,-\zeta , -\zeta^2 ,-\zeta , -\zeta^2 ,-\zeta , -\zeta^2) 
% \end{gather*}

This can be computed in the way demonstrated above: As $u^4$ is rationally conjugate to an element in \gapcom{3a} and $u^3$ is rationally conjugate to an element in \gapcom{2a}, we have $\chi_{\gapcom{10}}(u^4) = \chi_{\gapcom{10}}(\gapcom{3a}) = 1$ and $\chi_{\gapcom{10}}(u^3)= \chi_{\gapcom{10}}(\gapcom{2a})=2$. This gives %$D_{10}(u^4) \sim (1,\zeta,\zeta^2,1,\zeta,\zeta^2,1,\zeta,\zeta^2,1)$
$D_{\gapcom{10}}(u^4) \sim (\eigbox{4}{1}, \eigbox{3}{\zeta, \zeta^2})$ and 
%$D_{10}(u^3) \sim (1,1,1,1,1,1,-1,-1,-1,-1)$.
$D_{\gapcom{10}}(u^3) \sim (\eigbox{6}{1}, \eigbox{4}{-1} )$. Now $\chi_{\gapcom{10}}(u) = \varepsilon_{\gapcom{2a}}(u)\chi_{\gapcom{10}}(\gapcom{2a}) + \varepsilon_{\gapcom{3a}}(u)\chi_{\gapcom{10}}(\gapcom{3a}) = -1$ and as the eigenvalues of $D_{\gapcom{10}}(u)$ are products of the eigenvalues of $D_{\gapcom{10}}(u^4)$ and $D_{\gapcom{10}}(u^3)$, this gives the claimed eigenvalues. 

Moreover we have $\chi_{\gapcom{20}}(u^4)=\chi_{\gapcom{20}}(\gapcom{3a})=2$ and $\chi_{\gapcom{20}}(u^3)=\chi_{\gapcom{20}}(\gapcom{2a})=-4.$ So 
%\begin{align*}D_{20}(u^4) &\sim  (1,\zeta,\zeta^2,1,\zeta,\zeta^2,1,\zeta,\zeta^2,1,\zeta,\zeta^2,1,\zeta,\zeta^2,1,\zeta,\zeta^2,1,1)\\ \text{and }\quad D_{20}(u^3) &\sim (1,1,1,1,1,1,1,1-1,-1,-1,-1,-1,-1,-1,-1,-1,-1,-1,-1).\end{align*}
\[ D_{\gapcom{20}}(u^4) \sim (\eigbox{8}{1}, \eigbox{6}{\zeta, \zeta^2}) \quad \text{and} \quad D_{\gapcom{20}}(u^3) \sim (\eigbox{8}{1}, \eigbox{12}{-1}).\] Since $\chi_{\gapcom{20}}(u)=\varepsilon_{\gapcom{2a}}(u)\chi_{\gapcom{20}}(\gapcom{2a})+\varepsilon_{\gapcom{3a}}(u)\chi_{\gapcom{20}}(\gapcom{3a}) = 14$ we obtain the claimed eigenvalues.\\

As all the character values of all ordinary characters of $G$ are integers on all conjugacy classes of $G$, we may assume by a theorem of Fong \cite[Corollary 10.13]{Isaacs} that all ordinary representations mentioned above are $K$-representations, where $K$ is the $3$-adic completion of an extension of $\mathbb{Q}$ which is unramified at $3$. So if $R$ is the ring of integers of $K$ we may assume that they are even $R$-representations. Let $P$ be the maximal ideal of $R$ and let $\bar{\phantom{x}}$ denote the reduction modulo $P.$ Let $k$ be a finite field of characteristic 3 containing the residue class field of $R$ and affording all irreducible 3-modular representations of $M_{10}, \operatorname{PGL}(2,9)$ and $\Aut(A_6)$. Denote by $L_*$ an $RG$-lattice affording the representation $D_*$. Recall that $k$, $I(kC_3)$ and $kC_3$ denote the indecomposable $kC_6$-modules of $k$-dimension 1, 2 and 3 respectively having trivial composition factors and $(k)^-$, $I(kC_3)^-$ and $(kC_3)^-$ the indecomposable $kC_6$-modules of $k$-dimension 1, 2 and 3 respectively having non-trivial composition factors (see Propositions \ref{prop_modules_cyclic_mod_p} and \ref{prop_lattices_cyclic_group_of_prime_order}). 
We will write $T_*$ for a simple $kG$-module having character $\varphi_*.$ 

Regarded as $k\langle \bar{u} \rangle$-modules using Propositions \ref{prop_modules_cyclic_mod_p} and \ref{prop_lattices_cyclic_group} we may write $\bar{L}_* \cong \bar{L}_*^+ \oplus \bar{L}_*^{-}$ and $T_* \cong T_*^+ \oplus T_*^{-}$, where all the composition factors of $T^+_*$ and $\bar{L}_*^+$ are trivial and all the composition factors of $T^{-}_*$ and $\bar{L}_*^{-}$ are non-trivial. As $u^3$ is rationally conjugate to an element in $\gapcom{2a}$, it is also 3-adically conjugate to this element by \cite[Lemma 2.9]{HertweckAlgColloq}. Thus the $k$-dimensions of $T_*^+$ and $T_*^{-}$ can be deduced from the Brauer table above. The $k$-dimensions of $\bar{L}_*^+$ and $\bar{L}_*^{-}$ can be deduced from the eigenvalues given above using Proposition \ref{prop_lattices_cyclic_group}. The dimensions are given in Table \ref{dimensions_of_+_and_-_parts}.

\begin{table}[h]\centering
\begin{tabular}{l|c|c}
$k\langle \bar{u} \rangle$-module $T$ & $k$-dimension of $T^+$ & $k$-dimension of $T^{-}$ \\ \hline
\rule{0pt}{14pt}$\bar{L}_{\gapcom{10}}$ & 6 & 4 \\
$\bar{L}_{\gapcom{20}}$ & 8 & 12 \\
$T_{\gapcom{1*}}$ & 1 & 0 \\
$T_{\gapcom{6*}}$ & 2 & 4 \\
$T_{\gapcom{8}}$ & 4 & 4 \\
\end{tabular}\\[.4cm]
Where $\gapcom{*}$ takes all possible values in $\{\gapcom{a},\gapcom{b}\}.$\\
\caption{Dimensions of $T^+$ and $T^{-}$ for certain $k\langle \bar{u} \rangle$-modules}\label{dimensions_of_+_and_-_parts}
\end{table}

The Krull-Schmidt-Azumaya Theorem will be used without further mention. We will use decomposition series of $\bar{L}_*$ as $kG$-module, which we obtained using the \gapcom{GAP} package MeatAxe \cite{GAP4}\footnote{The representations of irreducible modules are available in \gapcom{GAP} by the command \gapcom{IrreducibleRepresentationsDixon} or  \gapcom{IrreducibleAffordingRepresentation} once the character is given.}, as shown in Table \ref{comp_fact_ZAut_A_6}.  

\begin{table}[h]\centering
\begin{tabular}{l|c|c}
$kG$-module $T$ & Socle of $T$ & Head of $T$ \\ \hline
\rule{0pt}{14pt}$\bar{L}_{\gapcom{10}}$ & $T_{\gapcom{1i}}$ & $T_{\gapcom{1j}}$ \\
$\bar{L}_{\gapcom{20}}$ & $T_{\gapcom{6a}} \oplus T_{\gapcom{6b}}$ & $T_{\gapcom{8}}$ \\
\end{tabular}\\[.4cm]
Where $(\gapcom{i},\gapcom{j})$ takes a value in $\{(\gapcom{a},\gapcom{b}),(\gapcom{b},\gapcom{a})\}$.\\
\caption{Decomposition factors of certain reduced $R \Aut(A_6)$-lattices}\label{comp_fact_ZAut_A_6}
\end{table}

In this paragraph all modules will be $k\langle \bar{u} \rangle$-modules. With the eigenvalues of $D_{\gapcom{20}}(u)$ as above using Propositions \ref{prop_lattices_cyclic_group} and \ref{prop_lattices_cyclic_group_of_prime_order} we get $\bar{L}_{\gapcom{20}}^{-} \cong 6I(kC_3)^-$. 
As $T_{\gapcom{1a}}$ and $T_{\gapcom{1b}}$ are trivial $k\langle \bar{u} \rangle$-modules by the Brauer table given above, using the eigenvalues of $D_{\gapcom{10}}(u)$ and Proposition \ref{prop_lattices_cyclic_group_of_prime_order}, we obtain $T^{-}_{\gapcom{8}} \cong \bar{L}_{\gapcom{10}}^{-} \cong X$ with $X \in \{(k)^- \oplus (kC_3)^-, 2(k)^- \oplus I(kC_3)^-\}$. But as $T^{-}_{\gapcom{8}} \cong \bar{L}_{\gapcom{20}}^{-}/(T_{\gapcom{6a}}^{-} \oplus T_{\gapcom{6b}}^{-})$, i.e.\ $T_{\gapcom{8}}^{-}$ is also a quotient of $\bar{L}_{\gapcom{20}}^{-} \cong 6I(kC_3)^-$, we get \[T^{-}_{\gapcom{8}} \cong 2(k)^- \oplus I(kC_3)^-.\]
So $6I(kC_3)^-/(T^{-}_{\gapcom{6a}} \oplus T^{-}_{\gapcom{6b}}) \cong \bar{L}_{\gapcom{20}}^{-}/(T^{-}_{\gapcom{6a}} \oplus T^{-}_{\gapcom{6b}}) \cong T^{-}_{\gapcom{8}} \cong 2(k)^- \oplus I(kC_3)^-$ and this implies 
\[T^{-}_{\gapcom{6a}} \oplus T^{-}_{\gapcom{6b}} \cong 2(k)^- \oplus 3I(kC_3)^-.\]
As $\dim_k(T_{\gapcom{6a}}^{-}) = \dim_k(T_{\gapcom{6b}}^{-}) = 4,$ this gives either 
\[T_{\gapcom{6a}}^{-} \cong 2(k)^- \oplus I(kC_3)^- \qquad \text{and} \qquad T_{\gapcom{6b}}^{-} \cong 2I(kC_3)^- \]
or 
\[T_{\gapcom{6a}}^{-} \cong 2I(kC_3)^- \qquad \text{and} \qquad T_{\gapcom{6b}}^{-} \cong 2(k)^- \oplus I(kC3)^-. \]

We will now apply restriction. First consider $u \in \V(\mathbb{Z}M_{10})$. 
\begin{table}[h]
\centering
\begin{tabular}{r|cc} 
 & \gapcom{1a} & \gapcom{2a} \\ \hline 
$\psi_{\gapcom{1a}}$ & $ 1 $ & $ 1 $ \\ 
 $\psi_{\gapcom{1b}}$ & $ 1 $ & $ 1 $ \\ 
 $\psi_{\gapcom{4a}}$ & $ 4 $ & $ . $ \\ 
 $\psi_{\gapcom{4b}}$ & $ 4 $ & $ . $  \\ 
 $\psi_{\gapcom{6}}$ & $ 6 $ & $ -2 $  \\
\end{tabular}
 \caption{Part of the Brauer table of $M_{10}$ for the prime $3,$ including all characters up to degree 6.}\label{brautab_M10}
\end{table}
Looking at the Brauer table of $\Aut(A_6)$ and the Brauer table of $M_{10}$ stated in Table \ref {brautab_M10} we obtain that $T_{\gapcom{6a}}$ and $T_{\gapcom{6b}}$ are isomorphic as $kM_{10}$-modules. So if $u$ lies in $\mathbb{Z}M_{10}$ they must also be isomorphic as $k\langle \bar{u} \rangle$-modules, contradicting the above.
 
Now assume $u$ lies in $\mathbb{Z}\PGL(2,9).$ Let $T$ be $T_{\gapcom{6a}}$ or $T_{\gapcom{6b}}$ such that as $k\langle \bar{u} \rangle$-module $T^{-} \cong 2(k)^- \oplus I(kC_3)^-.$ Looking at the Brauer table of $\PGL(2,9)$ given in Table \ref{brautab_PGL_2_9} we obtain that, considered as $k\PGL(2,9)$-module, $T$ has two $3$-dimensional composition factors, say $S_{\gapcom{3x}}$ and $S_{\gapcom{3y}}.$ Moreover by Clifford's theorem \cite[Theorem 11.1]{CurtisReiner1} $T$ is the direct sum of its two composition factors. The characters belonging to $S_{\gapcom{3x}}$ and $S_{\gapcom{3y}}$ both have the value $-1$ on \gapcom{2a}, so as $k\langle \bar{u} \rangle$-modules $S_{\gapcom{3x}}^{-}$ and $S_{\gapcom{3y}}^{-}$ are both 2-dimensional and thus one of them is isomorphic to $2(k)^-$ while the other one is isomorphic to $I(kC_3)^-$. 

\begin{table}[h] 
\centering
\begin{tabular}{r|cccccc} 
 & \gapcom{1a} & \gapcom{2a} & \gapcom{4a} & \gapcom{5a} & \gapcom{5b} & \gapcom{2b} \\ \hline 
$\tau_{\gapcom{1a}}$ & $ 1 $ & $ 1 $ & $ 1 $ & $ 1 $ & $ 1 $ & $ 1 $ \\ 
 $\tau_{\gapcom{1b}}$ & $ 1 $ & $ 1 $ & $ 1 $ & $ 1 $ & $ 1 $ & $ -1 $  \\ 
 $\tau_{\gapcom{3a}}$ & $ 3 $ & $ -1 $ & $ 1 $ & $ -\alpha $ & $ 1 + \alpha $ & $ -1 $  \\
 $\tau_{\gapcom{3b}}$ & $ 3 $ & $ -1 $ & $ 1 $ & $ 1 + \alpha $ & $ -\alpha $ & $ -1 $  \\
 $\tau_{\gapcom{3c}}$ & $ 3 $ & $ -1 $ & $ 1 $ & $ -\alpha $ & $ 1 + \alpha $ & $ 1 $  \\ 
 $\tau_{\gapcom{3d}}$ & $ 3 $ & $ -1 $ & $ 1 $ & $ 1 + \alpha $ & $ - \alpha $ & $ 1 $ \\ 
 $\tau_{\gapcom{4a}}$ & $ 4 $ & $ . $ & $ -2 $ & $ -1 $ & $ -1 $ & $ . $ \\ 
 $\tau_{\gapcom{4b}}$ & $ 4 $ & $ . $ & $ -2 $ & $ -1 $ & $ -1 $ & $ . $  \\ 
 \end{tabular} %\\[.4cm]
\\ \begin{tabular}{rl} with & $\alpha = \zeta_5 + \zeta_5^4$     \end{tabular}
\caption{Part of the Brauer table of $\PGL(2,9)$ for the prime $3$, including all characters up to degree 6.}\label{brautab_PGL_2_9}
\end{table}

Let $D_{\gapcom{3x}}: \PGL(2,9) \rightarrow \operatorname{GL}(3,k)$ be a representation of $\PGL(2,9)$ affording $S_{\gapcom{3x}}.$ Let $\alpha$ be the Frobenius automorphism of $k$ applied to every entry of a $3\times 3$-matrix over $k$. Then $\alpha \circ D_{\gapcom{3x}}$ is also a $k$-representation of $\PGL(2,9)$ and, looking at the Brauer table, we obtain that this representation affords $S_{\gapcom{3y}}.$ As $\alpha$ is linear on the full matrix ring, $S_{\gapcom{3x}}$ is also sent to $S_{\gapcom{3y}}$ as a $k\langle \bar{u} \rangle$-module via $\alpha$ and hence $S_{\gapcom{3x}}^{-}$ is sent to $S_{\gapcom{3y}}^{-}$ via $\alpha$. But, since $\alpha$ preserves the dimensions of eigenspaces of a matrix, $S_{\gapcom{3x}}^{-}$ and $S_{\gapcom{3y}}^{-}$ must in fact be isomorphic as $k\langle \bar{u} \rangle$-modules. This contradicts the above and thus the existence of $u$.\hfill $\Box$ 

\bigskip

\noindent \textbf{Acknowledgement:} We would like to thank Dr.\ Martin Hertweck for helpful conversations.

\bibliographystyle{amsalpha}
\bibliography{3primary_v5.bib}

\bigskip

Andreas Bächle, Vakgroep Wiskunde, Vrije Universiteit Brussel, Pleinlaan 2, B-1050 Brussels, Belgium.
\emph{ABachle@vub.ac.be}\\

Leo Margolis, Fachbereich Mathematik, Universit\"{a}t Stuttgart, Pfaffenwaldring 57, 70569 Stuttgart, Germany.
\emph{leo.margolis@mathematik.uni-stuttgart.de}

\end{document}